\documentclass{article}
\usepackage[utf8]{inputenc}
\usepackage[margin=1in]{geometry}
\usepackage{graphicx}
\usepackage{amssymb}
\usepackage{amsthm}
\usepackage{epstopdf}
\usepackage{dsfont}
\usepackage{xcolor}
\usepackage{amsmath}
\usepackage{mathrsfs}
\usepackage[parfill]{parskip}
\usepackage{comment}
\usepackage{tikz}
\usepackage{tikz-3dplot}
\usetikzlibrary{intersections}
\usepackage{mathtools}
\usepackage{hyperref}
\usepackage{caption}
\usepackage{mathrsfs}
\newtheorem{theorem}{Theorem}[section]

\newtheorem{lemma}[theorem]{Lemma}
\newtheorem{corollary}[theorem]{Corollary}
\theoremstyle{definition}
\newtheorem{definition}[theorem]{Definition}

\newtheorem{conjecture}[theorem]{Conjecture}

\title{Face Numbers of Shellable CW Balls and Spheres}
\author{Joshua Hinman}
\date{}

\begin{document}

\maketitle

\begin{abstract}
Let $\mathscr{X}$ be the boundary complex of a $(d+1)$-polytope, and let $\rho(d+1,k) = \frac{1}{2}[{\lceil (d+1)/2 \rceil \choose d-k} + {\lfloor (d+1)/2 \rfloor \choose d-k}]$. Recently, the author, answering B\'ar\'any's question from 1998, proved that for all $\lfloor \frac{d-1}{2} \rfloor \leq k \leq d$,
\[
    f_k(\mathscr{X}) \geq \rho(d+1,k)f_d(\mathscr{X}).
\]
We prove a generalization: if $\mathscr{X}$ is a shellable, strongly regular CW sphere or CW ball of dimension $d$, then for all $\lfloor \frac{d-1}{2} \rfloor \leq k \leq d$,
\[
    f_k(\mathscr{X}) \geq \rho(d+1,k)f_d(\mathscr{X}) + \frac{1}{2}f_k(\partial \mathscr{X}),
\]
with equality precisely when $k=d$ or when $k=d-1$ and $\mathscr{X}$ is simplicial. We further prove that if $\mathscr{S}$ is a strongly regular CW sphere of dimension $d$, and the face poset of $\mathscr{S}$ is both CL-shellable and dual CL-shellable, then $f_k(\mathscr{S}) \geq \min\{f_0(\mathscr{S}),f_d(\mathscr{S})\}$ for all $0 \leq k \leq d$.
\end{abstract}

\section{Introduction}

This paper is dedicated to the face numbers of shellable CW complexes. Our guiding question: if $\mathscr{X}$ is a shellable, strongly regular CW $d$-sphere or $d$-ball, and we are given $f_d(\mathscr{X})$, how small can its other face numbers be? Our work is motivated by similar questions about convex polytopes, so that is where we will begin our discussion.

B\'{a}r\'{a}ny asked the following in 1998 \cite{barany98}: if $P$ is a $d$-polytope and $0 \leq k \leq d-1$, can we guarantee that $f_k(P) \geq \min\{f_0(P),f_{d-1}(P)\}$? Despite its innocent appearance, B\'{a}r\'{a}ny's question remained open until 2023, when this author answered it in the affirmative \cite{hinman23}. The key result was a pair of linear inequalities:

\begin{theorem}[Hinman]
\label{polytopes}
    Let $P$ be a $(d+1)$-polytope and $0 \leq k \leq d$. Then
    \begin{align}
        f_k(P) &\geq \frac{1}{2}\biggl[{\lceil \frac{d+1}{2} \rceil \choose k} + {\lfloor \frac{d+1}{2} \rfloor \choose k}\biggr]f_0(P),\\
        f_k(P) &\geq \frac{1}{2}\biggl[{\lceil \frac{d+1}{2} \rceil \choose d-k} + {\lfloor \frac{d+1}{2} \rfloor \choose d-k}\biggr]f_d(P). \label{eqn_facets}
    \end{align}
\end{theorem}

The proof of Theorem \ref{polytopes} relied on solid angles. Of particular importance was a result by Perles and Shephard \cite[(23)]{perles67}: if $Q$ is a $d$-polytope, $0 \leq k \leq d-1$, and $\varphi_k(Q)$ is the solid angle sum of $Q$ at its $k$-dimensional faces, then
\begin{gather*}
    2\varphi_k(Q) \geq f_k(Q) - \max_\pi f_k(\pi(Q)),
\end{gather*}
where $\pi$ is an orthogonal projection onto a codimension-one hyperplane in general position. In other words, $2\varphi_k(Q)$ is bounded below by the minimum number of $k$-faces lost in a codimension-one projection. These lost faces comprise the interiors of two subcomplexes of $\partial Q$: the ``upper" complex and ``lower" complex induced by $\pi$. These induced subcomplexes were vital in proving a solid angle inequality \cite[Proposition 3.1]{hinman23}: for all $d$-polytopes $Q$ and all $0 \leq k \leq d-1$,
\begin{gather}
2\varphi_k(Q) \geq {\lceil \frac{d+1}{2} \rceil \choose d-k} + {\lfloor \frac{d+1}{2} \rfloor \choose d-k}. \label{eqn_angles}
\end{gather}
Inequality (\ref{eqn_facets}) results from summing (\ref{eqn_angles}) over all facets of a $(d+1)$-polytope $P$.

As some results about face numbers of polytopes continue to hold for larger classes of CW complexes, it is natural to ask a generalized version of B\'{a}r\'{a}ny's question: if $\mathscr{X}$ is a strongly regular CW $d$-sphere and $0 \leq k \leq d$, can we guarantee that $f_k(\mathscr{X}) \geq \min\{f_0(\mathscr{X}),f_d(\mathscr{X})\}$? Furthermore, does $\mathscr{X}$ satisfy the inequalities of Theorem \ref{polytopes}? This paper will partially answer each of the preceding questions.

Our main result is a generalization of (\ref{eqn_facets}): if $\mathscr{X}$ is a shellable, strongly regular CW sphere or ball of dimension $d \geq 1$, then for all $\lfloor \frac{d-1}{2} \rfloor \leq k \leq d$,
\begin{gather}
    f_k(\mathscr{X}) \geq \frac{1}{2} \left[{\lceil \frac{d+1}{2} \rceil \choose d-k} + {\lfloor \frac{d+1}{2} \rfloor \choose d-k}\right]f_d(\mathscr{X}) + \frac{1}{2}f_k(\partial \mathscr{X}), \label{eqn_mainresult}
\end{gather}
with equality precisely when $k=d$ or when $k=d-1$ and $\mathscr{X}$ is simplicial. Inequality (\ref{eqn_mainresult}) implies the following: if $\mathscr{S}$ is a strongly regular CW $d$-sphere such that the face poset $L(\mathscr{S})$ is both CL-shellable and dual CL-shellable, then for all $0 \leq k \leq d$,
\begin{align}
    f_k(\mathscr{S}) &\geq \frac{1}{2} \left[{\lceil \frac{d+1}{2} \rceil \choose k} + {\lfloor \frac{d+1}{2} \rfloor \choose k}\right]f_0(\mathscr{S}),\\
    f_k(\mathscr{S}) &\geq \frac{1}{2} \left[{\lceil \frac{d+1}{2} \rceil \choose d-k} + {\lfloor \frac{d+1}{2} \rfloor \choose d-k}\right]f_d(\mathscr{S}).
\end{align}
Consequently, $f_k(\mathscr{S}) \geq \min\{f_0(\mathscr{S}),f_d(\mathscr{S})\}$.

Bj\"{o}rner and Wachs proved that a graded poset is CL-shellable if and only if it admits a recursive atom ordering \cite{bjorner83}. Thus, when we require $L(\mathscr{S})$ to be both CL-shellable and dual CL-shellable, it is equivalent to saying that $L(\mathscr{S})$ admits both a recursive atom ordering and a recursive coatom ordering.

We would like to prove (\ref{eqn_mainresult}) in the same way as Theorem \ref{polytopes}---but since CW spheres do not have solid angles or nice orthogonal projections, we will need to modify our approach. Let $\mathscr{S}$ be a shellable, strongly regular CW $(d-1)$-sphere and $(F_1, \ldots, F_n)$ a shelling of its $(d-1)$-faces. In place of an ``upper" and ``lower" subcomplex, we will consider a ``beginning" subcomplex $\mathscr{C}=\langle F_1, \ldots, F_j\rangle$ and an ``ending" subcomplex $\mathscr{D}=\langle F_{j+1}, \ldots, F_n \rangle$ for some $1 \leq j \leq n$. We will mimic the proof of Theorem \ref{polytopes}, but in place of (\ref{eqn_angles}), we will show that
\begin{gather}
    f_k(\operatorname{int}\mathscr{C}) + f_k(\operatorname{int}\mathscr{D}) \geq {\lceil \frac{d+1}{2} \rceil \choose d-k} + {\lfloor \frac{d+1}{2} \rfloor \choose d-k}. \label{eqn_cells}
\end{gather}
Ultimately, we will prove inequality (\ref{eqn_mainresult}) by summing (\ref{eqn_cells}) over all $d$-faces of a CW $d$-sphere or CW $d$-ball $\mathscr{X}$.

\section{Preliminaries}
In this section, we will introduce the basic notions of regular, strongly regular, and shellable CW complexes. We will then discuss face posets, the combinatorial tools for describing a CW complex's structure. We refer the reader to \cite{bjorner82A,bjorner82B,bjorner83} for any undefined terminology.

\subsection{CW complexes}
We begin by introducing the main objects of this paper: strongly regular CW complexes.

\begin{definition}
\label{cwcomplex}
    A \emph{CW complex} of dimension $d \geq 0$, or \emph{CW $d$-complex}, is a set $\mathscr{X}$ of closed balls of dimension at most $d$ (i.e. topological balls homeomorphic to $\mathds{B}^k$, $0 \leq k \leq d$), together with a topological space $|\mathscr{X}|$ and a set of attaching maps defined recursively as follows.
    \begin{itemize}
        \item If $d=0$, then $\mathscr{X}$ is a finite set of points, and $|\mathscr{X}|$ is the set $\mathscr{X}$ under the discrete topology.
        \item If $d>0$, we construct $\mathscr{X}$ from a CW $(d-1)$-complex $\mathscr{X}'$. We let $\mathscr{X} = \mathscr{X}' \cup \{F_1, \ldots, F_n\}$, where $n \geq 0$ and $F_1, \ldots, F_n$ are closed $d$-balls. We then build the topological space $|\mathscr{X}|$ by attaching $F_1, \ldots, F_n$ to $|\mathscr{X}'|$ with attaching maps $\partial F_i \mapsto |\mathscr{X}'|$, $i=1, \ldots, n$.    
    \end{itemize}
    We say that $\mathscr{X}$ is \emph{regular} if the attaching map for each ball in $\mathscr{X}$ is a homeomorphism on the boundary of the ball.
    
    If $\mathscr{X}$ is a CW $d$-complex, we call the elements of $\mathscr{X} \cup \{\varnothing\}$ the \emph{faces} of $\mathscr{X}$. For all $0 \leq k \leq d$, we define $f_k(\mathscr{X})$ as the number of $k$-dimensional faces of $\mathscr{X}$; similarly, if $S$ is a subset of $\mathscr{X}$, we define $f_k(S)$ as the number of $k$-dimensional faces belonging to $S$.

    For any face $G \in \mathscr{X}$, we define the \emph{geometric realization} $|G|$ of $G$ as the image of $G$ in $|\mathscr{X}|$ under the inclusion map. For a set of faces $S \subseteq \mathscr{X}$, we define $|S| = \bigcup_{G \in S}|G|$.
\end{definition}

For faces $G_1, G_2 \in \mathscr{X}$, we will often say ``$G_1$ is contained in $G_2$" or ``$G_2$ contains $G_1$" as shorthand to mean $|G_1| \subseteq |G_2|$.

\begin{definition}
    Let $\mathscr{X}$ be a regular CW complex. We say that $\mathscr{X}$ is \emph{strongly regular} if for any two faces $G_1, G_2 \in \mathscr{X}$, $|G_1| \cap |G_2|$ is the geometric realization of a (possibly empty) face of $\mathscr{X}$. 
\end{definition}

We will primarily deal with CW complexes that realize spheres and balls. Introducing such objects will require a few more definitions.

\begin{definition}
    Let $\mathscr{X}$ be a strongly regular CW $d$-complex. We call $\mathscr{X}$ \emph{pure} if each face of $\mathscr{X}$ is contained in a face of dimension $d$. If, in addition, each $(d-1)$-face is contained in at most two $d$-faces, we call $\mathscr{X}$ a \emph{pseudomanifold}.
\end{definition}

\begin{definition}
    A \emph{CW $d$-sphere} (resp. \emph{$d$-ball}) is a strongly regular CW $d$-complex $\mathscr{X}$ such that $|\mathscr{X}| \cong \mathds{S}^d$ (resp. $\mathds{B}^d$).
\end{definition}
Note that any CW sphere or ball is a pseudomanifold.

The following two definitions deal with certain subsets of CW complexes.
\begin{definition}
    Let $\mathscr{X}$ be a CW complex and $S \subseteq \mathscr{X}$ a set of faces. The \emph{closure} of $S$, denoted $\langle S \rangle$, is the CW complex consisting of all faces of $\mathscr{X}$ which are contained in some member of $S$.
\end{definition}

\begin{definition}
    Let $\mathscr{X}$ be a $d$-dimensional pseudomanifold. The \emph{boundary complex} $\partial \mathscr{X}$ is the closure of the set of $(d-1)$-faces of $\mathscr{X}$ contained in exactly one $d$-face. The \emph{interior} $\operatorname{int}\mathscr{X}$ is the set of faces of $\mathscr{X}$ not contained in $\partial \mathscr{X}$.
\end{definition}
Note that $\operatorname{int}\mathscr{X}$ is not, in general, a CW complex.

If $G$ is a face of some strongly regular CW complex, we will often write $\partial G$ as shorthand for $\partial \langle G \rangle$. Henceforth, this is what $\partial G$ will mean unless stated otherwise.

Shellable CW complexes will play a central role in this paper. Our definition of shelling is taken from \cite{bjorner84}.
\begin{definition}
    Let $\mathscr{X}$ be a pure, regular CW $d$-complex. An ordering $(F_1,\ldots,F_n)$ of the $d$-faces of $\mathscr{X}$ is a \emph{shelling} if either $d=0$, or $d>0$ and
    \begin{itemize}
        \item $\partial F_j$ has a shelling for all $1 \leq j \leq n$.
        \item For all $1 < j \leq n$, $\partial F_j \cap \bigcup_{i=1}^{j-1} \partial F_i$ is a pure $(d-1)$-complex, and there exists a shelling of $\partial F_j$ beginning with the $(d-1)$-faces of $\partial F_j \cap \bigcup_{i=1}^{j-1} \partial F_i$.
    \end{itemize}
    A pure, regular CW complex is \emph{shellable} if it admits a shelling.
\end{definition}

\subsection{Face posets and lattices}
When working with regular CW complexes, it is often useful to consider the face poset. The face poset describes a CW complex's combinatorial structure; that is, which faces are contained in which. This section will introduce some basic properties of these posets.

Of particular interest are diamond lattices, a special type of lattices which include the face posets of CW spheres.

\begin{definition}
    Let $L$ be a finite, graded poset. We say $L$ is a \emph{lattice} if any two elements of $L$ have a unique least upper bound and a unique greatest lower bound. We say $L$ is a \emph{diamond lattice} if, in addition, each interval of $L$ with length two has size exactly four.
\end{definition}
Some sources \cite{bjorner84} also refer to a diamond lattice as a ``thin" lattice.

The following is a direct consequence of Proposition 3.1 in \cite{xue23}.
\begin{lemma}
\label{atom}
    Let $(L,\leq)$ be a diamond lattice of rank at least two. For every coatom $F$ of $L$, there exists an atom $v$ of $L$ such that $v \not \leq F$.
\end{lemma}

\begin{lemma}[Xue {\cite[Proposition 3.2]{xue23}}]
\label{boolean}
    Let $L = [\hat 0,\hat 1]$ be a diamond lattice of rank $d+1$ and $G$ an element of rank $r$. Then for all $r \leq k+1 \leq d$, the upper interval $[G, \hat 1]$ has at least ${d-r+1 \choose d-k}$ elements of rank $k+1$.
\end{lemma}

\begin{definition}
    Let $\mathscr{X}$ be a CW complex. The \emph{face poset} $L(\mathscr{X})$ is the set of faces of $\mathscr{X}$ ordered by containment, with a unique minimal element $\hat{0}$ representing the empty face, and with a unique maximal element $\hat{1}$ added.
\end{definition}
Note that for all pseudomanifolds $\mathscr{X}$, including all CW spheres, $L(\mathscr{X})$ is a diamond lattice.

The next three theorems, due to Bj\"{o}rner, show that all shelling information about a regular CW complex is encoded in its face poset. These theorems involve a \emph{CL-shellability}, a version of shellability for lattices. The precise definition of a CL-shellable lattice will never be used and is hence omitted, but can be found in \cite{bjorner82A}.

A lattice is \emph{dual CL-shellable} if its dual lattice is CL-shellable.

\begin{theorem}[Bj\"{o}rner {\cite[Proposition 4.2]{bjorner84}}]
\label{poset}
    Let $\mathscr{X}$ be a pure, regular CW complex. Then $\mathscr{X}$ is shellable if and only if $L(\mathscr{X})$ is dual CL-shellable.
\end{theorem}

\begin{theorem}[Bj\"{o}rner {\cite[Proposition 4.3]{bjorner84}}]
\label{intersection}
    Let $\mathscr{X}$ be a pure, strongly regular CW $d$-complex with a shelling $(F_1, \ldots, F_n)$. For all $1 < j \leq n$, $\partial F_j \cap \bigcup_{i=1}^{j-1} \partial F_i$ is a PL-ball or PL-sphere of dimension $d-1$. Furthermore, $\mathscr{X}$ is a CW ball if and only if each $\partial F_j \cap \bigcup_{i=1}^{j-1} \partial F_i$ is a PL-ball.
\end{theorem}

\begin{theorem}[Bj\"{o}rner {\cite[Proposition 4.5]{bjorner84}}]
\label{diamond}
    Let $L$ be a finite, graded poset of rank $d+2$ with a unique least element $\hat{0}$ and greatest element $\hat{1}$. The following are equivalent:
    \begin{itemize}
        \item $L$ is a dual CL-shellable diamond lattice.
        \item $L \cong L(\mathscr{X})$, $\mathscr{X}$ a shellable CW $d$-complex.
        \item $L \cong L(\mathscr{X})$, $\mathscr{X}$ a shellable CW $d$-sphere.
    \end{itemize}
\end{theorem}

Bj\"{o}rner's arguments for Theorem \ref{poset} imply a stronger statement: if $\mathscr{X}$ is a pure, regular CW $d$-complex with $d$-faces $F_1, \ldots, F_n$, then $L(\mathscr{X})$ fully determines whether $(F_1, \ldots, F_n)$ is a shelling. Thus, when discussing shellings of a CW complex $\mathscr{X}$, we need never worry about the precise attaching maps used to construct $|\mathscr{X}|$.

\section{Main results}
This section is dedicated to our main results on face numbers. We will begin with a topological lemma, used to relate shellings of a CW $d$-complex with shellings of an individual $d$-face's boundary (Lemma \ref{interior}). Next, we will prove a $d$-sphere analogue of (\ref{eqn_angles}) by induction on $d$ (Lemmas \ref{two_faces}, \ref{face_counting}). With these pieces in place, we will prove our lower bound on face numbers for shellable, strongly regular CW spheres and CW balls (Theorem \ref{lower_bound}). Finally, we will discuss conditions under which a CW $d$-sphere $\mathscr{S}$ must satisfy B\'{a}r\'{a}ny's property: $f_k(\mathscr{S}) \geq \min\{f_0(\mathscr{S}),f_d(\mathscr{S})\}$ for all $0 \leq k \leq d$.

\begin{lemma}
\label{interior}
    Let $\mathscr{X}$ be a strongly regular CW ball of dimension $d>0$ with a shelling $(F_1, \ldots, F_n)$. Let $\mathscr{C} = \partial F_n \cap \bigcup_{i=1}^{n-1} \partial F_i$. Then $\mathscr{C}$ is a CW $(d-1)$-ball and $\operatorname{int}\mathscr{C} \subseteq \operatorname{int}\mathscr{X}$.
\end{lemma}

\begin{proof}
    By Theorem \ref{intersection}, $\langle F_i \rangle_{i=1}^{n-1}$ is a CW $d$-ball and $\mathscr{C}$ is a CW $(d-1)$-ball. Let $A = \bigcup_{i=1}^{n-1}|F_i| \cong \mathds{B}^d$ and $B = |F_n| \cong \mathds{B}^d$. Then $A \cup B = |\mathscr{X}| \cong \mathds{B}^d$ and $A \cap B = |\mathscr{C}| \cong \mathds{B}^{d-1}$.
    
    Consider a point $x \in |\mathscr{C}|\backslash|\partial \mathscr{C}|$. By Mayer-Vietoris, the following is an exact sequence:
    \[
        \tilde H_d(A \backslash x) \oplus \tilde H_d(B \backslash x) \longrightarrow \tilde H_d(|\mathscr{X}|\backslash x) \longrightarrow \tilde H_{d-1}(|\mathscr{C}|\backslash x) \longrightarrow \tilde H_{d-1}(A \backslash x) \oplus \tilde H_{d-1}(B \backslash x).
    \]
    We know $x \in \partial A, \partial B$, so $A \backslash x \simeq B \backslash x \simeq \mathds{B}^d$ (where $\simeq$ denotes homotopy equivalence). Thus, our exact sequence becomes
    \[
        0 \longrightarrow \tilde H_d(|\mathscr{X}|\backslash x) \longrightarrow \tilde H_{d-1}(|\mathscr{C}|\backslash x) \longrightarrow 0.
    \]
    It follows that $\tilde H_d(|\mathscr{X}|\backslash x) \cong \tilde H_{d-1}(|\mathscr{C}|\backslash x) \cong \mathds{Z}$. Hence, $x \notin \partial|\mathscr{X}| = |\partial\mathscr{X}|$. We may conclude that $|\mathscr{C}|\backslash|\partial \mathscr{C}|$ and $|\partial \mathscr{X}|$ are disjoint.

    Let $G \in \operatorname{int}\mathscr{C}$. Then $|G| \cap (|\mathscr{C}|\backslash|\partial \mathscr{C}|) \neq \varnothing$, so $|G| \not \subseteq |\partial \mathscr{X}|$. Thus, $G \in \operatorname{int}\mathscr{X}$. This completes our proof that $\operatorname{int}\mathscr{C} \subseteq \operatorname{int}\mathscr{X}$.
\end{proof}

\begin{lemma}
\label{two_faces}
    Let $\mathscr{S}$ be a strongly regular CW sphere of dimension $d \geq 0$ with a shelling $(F_1, \ldots, F_n)$. Let $1 \leq j < n$, let $\mathscr{C} = \langle F_i \rangle_{i=1}^j$, and let $\mathscr{D} = \langle F_i \rangle_{i=j+1}^n$. Then there exist faces $C \in \operatorname{int}\mathscr{C}$ and $D \in \operatorname{int}\mathscr{D}$ such that $\dim C + \dim D \leq d$.
\end{lemma}

\begin{proof}
    We will use induction on $d$. For our base case, suppose $d=0$. Then $n=2$ and $j=1$, so $\mathscr{C} = \{C\}, \mathscr{D} = \{D\}$ for the two vertices $C, D$ of $\mathscr{S}$. We can see that $\dim C + \dim D = d = 0$.

    For our inductive step, fix $d \geq 1$ and suppose the statement holds for all strongly regular $(d-1)$-spheres. Let $\mathscr{S}$ be a strongly regular $d$-sphere with a shelling $(F_1,\ldots,F_n)$. Let $1 \leq j < n$, let $\mathscr{C} = \langle F_i \rangle_{i=1}^j$, and let $\mathscr{D} = \langle F_i \rangle_{i=j+1}^n$. Since $\mathscr{C},\mathscr{D}$ are pure, full-dimensional subcomplexes of a strongly regular CW sphere, $\mathscr{C},\mathscr{D}$ must be pseudomanifolds, so $\partial \mathscr{C},\partial \mathscr{D}$ and $\operatorname{int} \mathscr{C}, \operatorname{int} \mathscr{D}$ are well-defined. Furthermore, $\operatorname{int}\mathscr{C} = \mathscr{S}\backslash\mathscr{D}$ and $\operatorname{int}\mathscr{D} = \mathscr{S}\backslash\mathscr{C}$.

    If $j=1$, then $\operatorname{int}\mathscr{D}$ must include at least one vertex by Lemma \ref{atom}. Let $C = F_1$ and let $D$ be a vertex in $\operatorname{int}\mathscr{D}$. Then $\dim C + \dim D = d$, as desired.

    Now, suppose $j>1$. Let $Q$ be the set of $(d-1)$-faces of $\partial F_j$, and define subcomplexes $\mathscr{C}',\mathscr{D}'$ of $\partial F_j$ as follows:
    \begin{align*}
        \mathscr{C}' &= \langle R \in Q \mid R \in \partial F_1, \ldots, \text{or } \partial F_{j-1}\rangle,\\
        \mathscr{D}' &= \langle R \in Q \mid R \in \partial F_{j+1}, \ldots, \text{or } \partial F_n\rangle.
    \end{align*}
    Since $F_1, \ldots, F_n$ is a shelling of $\mathscr{S}$, there exists a shelling of $\partial F_j$ beginning with the $(d-1)$-faces of $\mathscr{C}'$. Thus, by our inductive hypothesis, there exist faces $C \in \operatorname{int}\mathscr{C}'$ and $D' \in \operatorname{int}\mathscr{D}'$ such that $\dim C + \dim D' \leq d-1$.
    
    By Theorem \ref{intersection}, $\mathscr{C}$ is a CW $d$-ball and $\partial F_j \cap \bigcup_{i-1}^{j-1} \partial F_i \supseteq \mathscr{C}'$ is a CW $(d-1)$-ball. Thus, $\mathscr{C}' = \partial F_j \cap \bigcup_{i-1}^{j-1} \partial F_i$. It follows by Lemma \ref{interior} that $C \in \operatorname{int}\mathscr{C}$.
    
    We can see that $\mathscr{C}' \cap \mathscr{D}' = \partial \mathscr{C}' = \partial \mathscr{D}'$, so $D' \notin \mathscr{C}'$; thus, $F_j$ is the only $d$-face of $\mathscr{C}$ containing $D'$. Consider the upper interval $[D',\hat 1]$ in $L(\mathscr{S})$. Since $L(\mathscr{S})$ is a diamond lattice (Lemma \ref{diamond}), $[D',\hat 1]$ must be a diamond lattice as well. Thus, by Lemma \ref{atom}, there exists an atom $D$ of $[D',\hat 1]$ (i.e. a face $D$ containing $D'$ with $\dim D = \dim D' + 1$) such that $D \notin \langle F_j \rangle$. This implies that $D \notin \mathscr{C}$, so $D \in \operatorname{int} \mathscr{D}$.

    Since $\dim C + \dim D' \leq d-1$, we know $\dim C + \dim D \leq d$. This completes our inductive step.    
\end{proof}

\begin{lemma}
\label{face_counting}
    Let $\mathscr{S}$ be a strongly regular CW sphere of dimension $d-1 \geq 0$ with a shelling $(F_1,\ldots,F_n)$. Let $0 \leq j \leq n$, let $\mathscr{C} = \langle F_i \rangle_{i=1}^j$, and let $\mathscr{D} = \langle F_i \rangle_{i=j+1}^n$. Then for all $\lfloor \frac{d-1}{2} \rfloor \leq k \leq d-1$,
    \[
        f_k(\operatorname{int}\mathscr{C}) + f_k(\operatorname{int}\mathscr{D}) \geq {\lceil \frac{d+1}{2} \rceil \choose d-k} + {\lfloor \frac{d+1}{2} \rfloor \choose d-k}.
    \]
\end{lemma}

\begin{proof}
    First, suppose $j=0$ or $j=n$. If $j=0$, then $\mathscr{C}=\varnothing$ and $\mathscr{D} = \mathscr{S}$; likewise, if $j=n$, then $\mathscr{C} = \mathscr{S}$ and $\mathscr{D} = \varnothing$. Thus,
    \[
        f_k(\operatorname{int}\mathscr{C}) + f_k(\operatorname{int}\mathscr{D}) = f_k(\mathscr{S}).
    \]
    By Lemma \ref{boolean}, we know
    \[
        f_k(\mathscr{S}) \geq {d+1 \choose d-k} \geq {\lceil \frac{d+1}{2} \rceil \choose d-k} + {\lfloor \frac{d+1}{2} \rfloor \choose d-k},
    \]
    so we are done.

    Suppose instead that $0 < j < n$. Then by Lemma \ref{two_faces}, there exist faces $C \in \operatorname{int} \mathscr{C}$ and $D \in \operatorname{int} \mathscr{D}$ such that $\dim C + \dim D \leq d-1$. We know $C$ has rank $\dim C + 1$ in the face poset $L(\mathscr{S})$, so by Lemma \ref{boolean}, there are at least ${d-\dim C \choose d-k}$ $k$-faces of $\mathscr{S}$ containing $C$. Thus,
    \[
        f_k(\operatorname{int}\mathscr{C}) \geq {d-\dim C \choose d-k}.
    \]
    By a similar argument,
    \[
        f_k(\operatorname{int}\mathscr{D}) \geq {d-\dim D \choose d-k}.
    \]
    We know that $(d-\dim C) + (d-\dim D) \geq d+1$. Hence, by Lemma 2.11 of \cite{hinman23},
    \[
        f_k(\operatorname{int}\mathscr{C}) + f_k(\operatorname{int}\mathscr{D}) \geq {\lceil \frac{d+1}{2} \rceil \choose d-k} + {\lfloor \frac{d+1}{2} \rfloor \choose d-k}. \qedhere
    \]
\end{proof}

\begin{theorem}
\label{lower_bound}
    Let $\mathscr{X}$ be a shellable, strongly regular CW sphere or CW ball of dimension $d \geq 1$. Then for all $\lfloor \frac{d-1}{2} \rfloor \leq k \leq d$,
    \[
        f_k(\mathscr{X}) \geq \frac{1}{2} \left[{\lceil \frac{d+1}{2} \rceil \choose d-k} + {\lfloor \frac{d+1}{2} \rfloor \choose d-k}\right]f_d(\mathscr{X}) + \frac{1}{2}f_k(\partial \mathscr{X}),
    \]
    with equality precisely when $k=d$ or when $k=d-1$ and $\mathscr{X}$ is simplicial.
\end{theorem}

\begin{proof}
    First, observe that if $k=d$, then
    \[
        f_k(\mathscr{X}) = \frac{1}{2} \left[{\lceil \frac{d+1}{2} \rceil \choose d-k} + {\lfloor \frac{d+1}{2} \rfloor \choose d-k}\right]f_d(\mathscr{X}) + \frac{1}{2}f_k(\partial \mathscr{X}).
    \]
    For the remainder of this proof, we will assume that $k < d$.

    Let $n = f_d(\mathscr{X})$, and let $(F_1,\ldots,F_n)$ be a shelling of $\mathscr{X}$. Consider an arbitrary $1 \leq j \leq n$, and let $Q$ be the set of $(d-1)$-faces of $\partial F_j$. Define pure subcomplexes $\mathscr{C}_j, \mathscr{D}_j$ of $\partial F_j$ as follows:
    \begin{align*}
        \mathscr{C}_j &= \langle R \in Q \mid R \in \partial F_1, \ldots, \text{or } \partial F_{j-1} \rangle,\\
        \mathscr{D}_j &= \langle R \in Q \mid R \in \partial F_{j+1}, \ldots, \partial F_n, \text{or } \partial \mathscr{X} \rangle.
    \end{align*}
    We can see that $\mathscr{C}_j \cup \mathscr{D}_j = \partial F_j$. Furthermore, since $\mathscr{X}$ is a pseudomanifold, the sets of $(d-1)$-faces of $\mathscr{C}_j$ and $\mathscr{D}_j$ are disjoint. In other words, $\mathscr{C}_j \cap \mathscr{D}_j = \partial \mathscr{C}_j = \partial \mathscr{D}_j$.
    
    If $j=1$, then $\mathscr{C}_j = \partial F_j \cap \bigcup_{i=1}^{j-1} \partial F_{j-1} = \varnothing.$ If $j>1$, then by Theorem \ref{intersection}, $\partial F_j \cap \bigcup_{i=1}^{j-1} \partial F_{j-1}$ is a PL-ball or PL-sphere of dimension $d-1$. Since $\mathscr{C}_j$ contains all of the $(d-1)$-faces of $\partial F_j \cap \bigcup_{i=1}^{j-1} \partial F_{j-1}$, it follows that
    \[
        \mathscr{C}_j = \partial F_j \cap \bigcup_{i=1}^{j-1} \partial F_{j-1}
    \]
    for all $1 \leq j \leq n$.
    
    We now turn our attention to $\mathscr{D}_j$. By definition, if $j=1$, then $\mathscr{D}_j = \partial F_j \cap (\partial \mathscr{X} \cup \bigcup_{i=j+1}^n \partial F_i) = \partial F_j$. Likewise, if $j=n$, then $\mathscr{D}_j = \partial F_j \cap (\partial \mathscr{X} \cup \bigcup_{i=j+1}^n \partial F_i) = \partial F_j \cap \partial \mathscr{X}$. We claim that for all $1 \leq j \leq n$, $\mathscr{D}_j = \partial F_j \cap (\partial \mathscr{X} \cup \bigcup_{i=j+1}^n \partial F_i)$.

    Let $1 < j < n$. By definition, $\mathscr{D}_j \subseteq \partial F_j \cap (\partial \mathscr{X} \cup \bigcup_{i=j+1}^n \partial F_i)$. Furthermore, by Lemma \ref{interior}, $\partial F_j \backslash \mathscr{D}_j = \operatorname{int} \mathscr{C}_j \subseteq \operatorname{int} \langle F_i \rangle_{i=1}^j$. It follows that $\partial F_j\backslash \mathscr{D}_j$ is disjoint from $\partial F_{j+1}, \ldots, \partial F_n, \partial \mathscr{X}$, so $\partial F_j \backslash \mathscr{D}_j \subseteq \partial F_j \backslash (\partial \mathscr{X} \cup \bigcup_{i=j+1}^n \partial F_i)$. Thus,
    \[
        \mathscr{D}_j = \partial F_j \cap \left(\partial \mathscr{X} \cup \bigcup_{i=j+1}^n \partial F_i\right).
    \]
    
    We are ready to prove our main inequality. Let $G \in \mathscr{X}$. If $G \in \operatorname{int}\mathscr{C}_j$ for any $1 \leq j \leq n$, then $G \notin \mathscr{D}_j$, so $G \notin \partial F_{j+1}, \ldots, \partial F_n, \partial \mathscr{X}$. Likewise, if $G \in \operatorname{int}\mathscr{D}_j$ for any $1 \leq j \leq n$, then $G \notin \mathscr{C}_j$, so $G \notin \partial F_1, \ldots, \partial F_{j-1}$. That is, $G$ belongs to at most one of $\operatorname{int}\mathscr{C}_1, \ldots, \operatorname{int}\mathscr{C}_n, \partial\mathscr{X}$ and at most one of $\operatorname{int}\mathscr{D}_1, \ldots, \operatorname{int}\mathscr{D}_n$. Thus, for all $0 \leq k < d$,
    \[
        \sum_{j=1}^n [ f_k(\operatorname{int}\mathscr{C}_j) + f_k(\operatorname{int}\mathscr{D}_j) ] \leq 2f_k(\mathscr{X}) - f_k(\partial \mathscr{X}).
    \]
    Meanwhile, if $\lfloor \frac{d-1}{2} \rfloor \leq k < d$, then by Lemma \ref{face_counting},
    \[
        \sum_{j=1}^n [ f_k(\operatorname{int}\mathscr{C}_j) + f_k(\operatorname{int}\mathscr{D}_j) ] \geq \left[ {\lceil \frac{d+1}{2} \rceil \choose d-k} + {\lfloor \frac{d+1}{2} \rfloor \choose d-k} \right] n.
    \]
    We may conclude that for all $\lfloor \frac{d-1}{2} \rfloor \leq k \leq d$,
    \begin{gather}
        f_k(\mathscr{X}) \geq \frac{1}{2} \left[{\lceil \frac{d+1}{2} \rceil \choose d-k} + {\lfloor \frac{d+1}{2} \rfloor \choose d-k}\right]n + \frac{1}{2}f_k(\partial \mathscr{X}). \label{eqn_lowerbound}
    \end{gather}

    Finally, we will prove that equality holds in (\ref{eqn_lowerbound}) if and only if $k=d-1$ and $\mathscr{X}$ is simplicial. Suppose $k=d-1$ and $\mathscr{X}$ is simplicial. Then
    \[
        \frac{1}{2} \left[{\lceil \frac{d+1}{2} \rceil \choose d-k} + {\lfloor \frac{d+1}{2} \rfloor \choose d-k}\right]n + \frac{1}{2}f_k(\partial \mathscr{X}) = \left(\frac{d+1}{2}\right)n + \frac{1}{2}f_{d-1}(\partial \mathscr{X}).
    \]
    Each $d$-face of $\mathscr{X}$ is a simplex with exactly $d+1$ faces of dimension $d-1$. Meanwhile, each interior $(d-1)$-face of $\mathscr{X}$ is contained in two $d$-faces, and each boundary $(d-1)$-face is contained in one $d$-face. Thus, the sum $(d+1)n+f_{d-1}(\partial \mathscr{X})$ double-counts each $(d-1)$-face of $\mathscr{X}$. It follows that
    \[
        \left(\frac{d+1}{2}\right)n + \frac{1}{2}f_{d-1}(\partial \mathscr{X}) = f_{d-1}(\mathscr{X}),
    \]
    so we are done.

    Conversely, suppose equality holds in (\ref{eqn_lowerbound}). Then for all $1 \leq j \leq n$,
    \begin{gather}
        f_k(\operatorname{int} \mathscr{C}_j) + f_k(\operatorname{int} \mathscr{D}_j) = {\lceil \frac{d+1}{2} \rceil \choose d-k} + {\lfloor \frac{d+1}{2} \rfloor \choose d-k}, \label{eqn_equality}
    \end{gather}
    and in particular,
    \[
        f_k(\partial F_1) = {\lceil \frac{d+1}{2} \rceil \choose d-k} + {\lfloor \frac{d+1}{2} \rfloor \choose d-k}.
    \]
    Furthermore, by Lemma \ref{boolean} and Vandermonde's identity\footnote{a misnomer: Zhu Shijie stated the identity in 1303, some 470 years before Vandermonde. See \cite[pp. 59--60]{askey75}.},
    \[
        f_k(\partial F_1) \geq {d+1 \choose d-k} = {\lceil \frac{d+1}{2} \rceil \choose d-k} + {\lfloor \frac{d+1}{2} \rfloor \choose d-k} + \sum_{i=1}^{d-k-1} {\lceil \frac{d+1}{2} \rceil \choose i}{\lfloor \frac{d+1}{2} \rfloor \choose d-k-i}.
    \]
    It follows that
    \[
        \sum_{i=1}^{d-k-1} {\lceil \frac{d+1}{2} \rceil \choose i}{\lfloor \frac{d+1}{2} \rfloor \choose d-k-i} = 0,
    \]
    which can only occur if $k=d-1$.

    For each $1 \leq j \leq n$, equation (\ref{eqn_equality}) now reduces to
    \begin{align*}
        f_{d-1}(\operatorname{int}\mathscr{C}_j) + f_{d-1}(\operatorname{int}\mathscr{D}_j) &= d+1.\\
        \Rightarrow f_{d-1}(\partial F_j) &= d+1.
    \end{align*}
    Thus, by \cite[Proposition 3.3]{xue23}, each $F_j$ is a simplex. This completes our treatment of equality and our proof.    
\end{proof}

For any graded poset $L$ of rank $d+2$ and $0 \leq k \leq d$, let $f_k(L)$ be the number of elements of $L$ with rank $k+1$ (so for a pure, regular CW $d$-complex $\mathscr{X}$, $f_k(L(\mathscr{X})) = f_k(\mathscr{X}))$. The following are immediate consequences of Theorems \ref{lower_bound} and \ref{diamond}.

\begin{corollary}
    Let $L$ be a diamond lattice of rank $d+2$. If $L$ is CL-shellable, then for all $0 \leq k \leq \lceil \frac{d+1}{2} \rceil$,
    \[
        f_k(L) \geq \frac{1}{2}\left[{\lceil \frac{d+1}{2} \rceil \choose k} + {\lfloor \frac{d+1}{2} \rfloor \choose k}\right]f_0(L).
    \]
    If $L$ is dual CL-shellable, then for all $\lfloor \frac{d-1}{2} \rfloor \leq k \leq d$,
    \[
        f_k(L) \geq \frac{1}{2}\left[{\lceil \frac{d+1}{2} \rceil \choose d-k} + {\lfloor \frac{d+1}{2} \rfloor \choose d-k}\right]f_d(L).
    \]
\end{corollary}

\begin{corollary}
    Let $L$ be a diamond lattice of rank $d+2$ which is both CL-shellable and dual CL-shellable. Then for all $0 \leq k \leq d$, $f_k(L) \geq \min\{f_0(L),f_d(L)\}$.
\end{corollary}

\section{Concluding remarks}
Recall our guiding question: if $\mathscr{X}$ is a shellable, strongly regular CW $d$-sphere or CW $d$-ball, and we are given $f_d(\mathscr{X})$, how small can its other face numbers be? We have narrowed down the possibilities, but we still do not have a full answer.

Theorem \ref{polytopes}'s linear bounds for polytopes are tight: for any $d$ and $k$, we can find a family of neighborly (resp. dual neighborly) $(d+1)$-polytopes $P$ with $f_k(P)/f_d(P)$ (resp. $f_k(P)/f_0(P)$) asymptotically approaching the given coefficients \cite{hinman24}. Similarly, Theorem \ref{lower_bound} gives a tight linear bound on $f_k(\mathscr{X})-\frac{1}{2}f_k(\partial \mathscr{X})$. The boundary complexes of the aforementioned $(d+1)$-polytopes are a family of shellable $d$-spheres $\mathscr{X}$ with $f_k(\mathscr{X})/f_d(\mathscr{X})$ asymptotically approaching $\frac{1}{2}[{\lceil(d+1)/2\rceil \choose d-k}+{\lfloor(d+1)/2\rfloor \choose d-k}]$. The same complexes, minus one $d$-face each, are a family of shellable $d$-balls $\mathscr{X}$ with $[f_k(\mathscr{X})-\frac{1}{2}f_k(\partial \mathscr{X})]/f_d(\mathscr{X})$ asymptotically approaching $\frac{1}{2}[{\lceil(d+1)/2\rceil \choose d-k}+{\lfloor(d+1)/2\rfloor \choose d-k}]$. In other words, for any $d$ and $k$, the inequality of Theorem \ref{lower_bound} would no longer hold if we replaced $\frac{1}{2}[{\lceil(d+1)/2\rceil \choose d-k}+{\lfloor(d+1)/2\rfloor \choose d-k}]$ with a higher number.

This does not preclude the possibility of tighter \emph{nonlinear} bounds. A notable nonlinear bound on face numbers is Kalai's Generalized Upper Bound Theorem for simplicial polytopes \cite{kalai91}:
\begin{theorem}[Generalized Upper Bound Theorem]
\label{gubt}
    Let $C(d,n)$ be the cyclic $d$-polytope on $n$ vertices, and let $P$ be a simplicial $d$-polytope with $f_{d-1}(P) \geq f_{d-1}(C(d,n))$. Then $f_k(P) \geq f_k(C(d,n))$ for $k=0,\ldots,d-1$.
\end{theorem}
Kalai proved Theorem \ref{gubt} using the $g$-theorem and the theory of algebraic shifting. He then made further conjectures at varying levels of generality. The following two remain open:
\begin{conjecture}[Kalai]
\label{gubc_polytopes}
    The Generalized Upper Bound Theorem applies to arbitrary polytopes.
\end{conjecture}
\begin{conjecture}[Kalai]
\label{gubc_diamonds}
    The Generalized Upper Bound Theorem applies to arbitrary Eulerian lattices.
\end{conjecture}
In light of our results, we will squeeze our own conjecture in between \ref{gubc_polytopes} and \ref{gubc_diamonds}:
\begin{conjecture}
    The Generalized Upper Bound Theorem applies to arbitrary shellable, strongly regular CW spheres.
\end{conjecture}

\section{Acknowledgements}
The author would like to thank Isabella Novik for her extensive guidance in writing and editing this paper. The author would also like to thank Hailun Zheng for her help editing. The author was  partially supported by a graduate fellowship from NSF grant DMS-2246399.

\bibliography{bibliography}
\bibliographystyle{plain}

\end{document}